\documentclass[11pt,reqno,fleqn]{amsart}

\usepackage{amsfonts,amsmath,amsthm,amssymb}
\usepackage{setspace}

\newtheorem{thm}{Theorem}[section]
\newtheorem{defin}{Definition}[section]
\newtheorem{lemma}{Lemma}[section]
\newtheorem{prop}{Proposition}[section]

\newtheorem*{thmA}{Theorem A}
\newtheorem*{thmB}{Theorem B}

\newcommand{\R}{\ensuremath{\mathbb{R}}}
\newcommand{\Z}{\ensuremath{\mathbb{Z}}}

\newcommand{\N}{\ensuremath{\mathbb{N}}}
\newcommand{\PP}{\ensuremath{\mathbb{P}}}
\newcommand{\EE}{\ensuremath{\mathbb{E}}}
\newcommand{\1}{\ensuremath{\mathbf{1}}}

\newcommand{\subs}{\ensuremath{\subseteq}}

\newcommand{\eps}{\ensuremath{\varepsilon}}
\newcommand{\la}{\ensuremath{\lambda}}
\newcommand{\La}{\ensuremath{\Lambda}}
\newcommand{\al}{\ensuremath{\alpha}}

\newcommand{\de}{\ensuremath{\delta}}

\newcommand{\eq}{\begin{equation}
\newcommand{\ee}{\end{equation}}}

\numberwithin{equation}{section}

\begin{document}
\title{Constellations in $\PP^d$}
\author{Brian Cook\quad\quad\quad\'Akos Magyar}
\thanks{The second author was supported by NSERC Grant 22R44824.}

\address{Department of Mathematics, The University of British Columbia, Vancouver, BC, V6T1Z2, Canada}
\email{bcook@math.ubc.ca}
\address{Department of Mathematics, University of British Columbia, Vancouver, B.C. V6T 1Z2, Canada}
\email{magyar@math.ubc.ca}

\begin{abstract} Let $A$ be a subset of positive relative upper density of $\PP^d$, the $d$-tuples of primes. We prove that $A$ contains an affine copy of any set $e\subs\Z^d$, as long as $e$ is in \emph{general position} in the sense that the set $e\cup\{0\}$ has at most one point on every coordinate hyperplane.
\end{abstract}

\maketitle
%\doublespacing
%%%%%%%%%%%%%%%%%%%%%%%%%%%%%%%%%%%%%%%%%%%%%%%%%%%%%%%%%%%%%%%%%%%%

\section{Introduction.}
\subsection{Background}
The celebrated theorem of Green and Tao \cite{GT} states that subsets of positive relative upper density of the primes contain an affine copy of any finite set of the integers, in particular contain arbitrary long arithmetic progressions. It is natural to ask if similar results hold in the multi-dimensional settings, especially in light of the multi-dimensional extensions of the closely related theorem of Szemer\'{e}di \cite{SZ} on arithmetic progressions in dense subsets of the integers. Indeed such a result was obtained by Tao \cite{T}, showing that the Gaussian primes contain arbitrary constellations. In the same paper the problem of finding constellations in dense subsets of $\PP^d$ was raised and briefly discussed.

The difficulty in this settings comes from two facts. First, the natural majorant of the $d$-tuples of primes is not pseudo-random with respect to the box norms, which replace the Gowers' uniformity norms in the multi-dimensional case. This may be circumvented by assuming the set $e$ is in \emph{general position} as described below, as is already suggested in \cite{T}. However even under the this non-degeneracy assumption, the so-called \emph{correlation conditions} in \cite{GT} do not seem to be sufficient, and a key observation of this note is to use more general correlation conditions to obtain the dual function estimates in the multi-dimensional case. Also, we need to use an abstract \emph{transference principle} due to Gowers \cite{G2} and independently to Reingold, Trevisan, Tulsiani and Vadhan \cite{RTTV}, see also Tao and Ziegler \cite{TZ}.

\subsection{Main Results.}

Let $e=\{e_1,\ldots,e_l\}\in (\Z^d)^l$ be a set of vectors; a constellation defined by $e$ is then a set $e'=\{x,x+te_1,\ldots,x+te_l\}$ where $t\neq 0$ is a scalar, that is and affine image of the set $e\cup\{0\}$.

\begin{defin} We say that a set of $l$ vectors $e\in (\Z^d)^l$ is in general position, if $|\pi_i(e\cup\{0\})|=l+1$ for each $i$, where $\pi_i$ is the orthogonal projection to the $i^{th}$ coordinate axis.
\end{defin}

Let us also recall that a subset $A$ of the $d$-tuples of primes $\PP^d$ is of positive upper relative density if
\[\limsup_{N\to\infty}\ \frac{|A\cap [1,N]^d|}{\pi(N)^d}\ >0\]
Our main result is then the following

\begin{thm} \label{prime} Given any set $A\subs\PP^d$ of positive relative upper density, we have that $A$ contains infinitely many constellations defined by a set of vectors $e\in (\Z^d)^l$ in general position.
\end{thm}

\noindent\textbf{Remarks:} We note that for $d=1$ this translates back the above described theorem of Green and Tao \cite{GT}, as any finite subset of $\Z$ is in general position.

Also, one may assume that $l=d$ and the set $e=\{e_1,\ldots ,e_d\}\subs\Z^d$ forms a basis in $\R^d$ besides being in general position, by passing to higher dimensions. Indeed, if $e\in(\Z^d)^l$ then let $\{f_1,\ldots ,f_l\}\subs\Z^l$ be linearly independent vectors, and define a basis \newline $e'=\{e'_1=(e_1,f_1),\ldots,e'_l=(e_l,f_l),e'_{l+1},\ldots ,e'_{l+d}\}\subs \Z^{d+l}$ by extending the linearly independent set of vectors $e'_i=(e_i,f_i),\ (1\leq i\leq l)$. If $e$ was in general position then it is easy to make the construction so that $e'$ is also in general position, and if the set $A':=A\times \PP^l$ contains a constellation $x'+te'$, then $A$ contains $x+te$. Thus from now on we will always assume that $e$ is also a basis of $\R^d$.

Theorem \ref{prime} may be viewed as a relative version of the so-called Multidimensional Szemer\'{e}di Theorem \cite{FK}, stating that any subset of $\Z^d$ of positive upper density contains infinitely many constellations defined by any finite set of vectors $e\subs\Z^d$. As is customary, we will work in the "finitary" settings, when the underlying space is the group $\Z_N^d=(\Z/N\Z)^d$, $N$ being a large prime. In this settings we need the following, more quantitative version:

\begin{thmA} [Furstenberg-Katznelson \cite{FK}]  Let $\al>0$, $d\in\N$ and let $e=\{e_1,\ldots,e_d\}\subs\Z_N^d$ be a fixed set of vectors.
If $f:\Z_N^d\to [0,1]$ is a given function such that $\EE (f(x):\,x\in\Z_N^d)\geq\al$, then one has
\eq\EE(f(x)f(x+te_1)\ldots f(x+te_d):\ x\in\Z_N^d,\,t\in\Z_N)\geq c(\al,e)\ee
where $c(\al,e)>0$ is a constant depending only on $\al$ and the set $e$.
\end{thmA}

Here we used the "expectation" notation: $\EE(f(x):\ x\in A)=\frac{1}{|A|}\sum_{x\in A} f(x)$.

In the relative settings, when $A\subs\PP^d$, the condition: $\ \EE (f(x): x\in\Z_N^d)\geq\al$ (after identifying $[1,N]^d$ with $\Z_N^d$) does not hold for the indicator function $f=\1_A$, however it holds for $f=\1_A \La^d$ where $\La^d$ is the $d$-fold tensor product of the von Mangoldt function $\La$. The price one pays is that the function $f$ is no longer bounded uniformly in $N$. Following the strategy of \cite{GT} we will show that the $d$-fold tensor product $\otimes^d\nu$ of the pseudo-random measure $\nu$ used in \cite{GT} is sufficiently random in our settings in order to apply the transference principle of \cite{G2}; we will refer to such measures $\nu$ as \emph{$d$-pseudo-random} measures. We postpone the definition of $d$-pseudo-random measures to the next section, but state our main result in the finitary settings below:

\begin{thm}\label{finite}
Let $\alpha>0$ be given, and $d$ be fixed. There exists a constant
$c(\alpha,e)>0$ such that the following holds. If $0\leq f \leq \mu$ is a given function on $\Z_N^d$ such that $\mu=\otimes^d\nu$ where $\nu$ is
$d$-pseudo-random, and $\mathbb{E}(f(x):\,x\in\Z_N^d)\geq \alpha$, then for any basis
$e=\{e_1,...,e_d\}$ in general position, we have that
\eq
\mathbb{E}(f(x)f(x+te_1)...f(x+te_d):x\in\mathbb{Z}_N^d,t\in\mathbb{Z}_N)\geq
c(\al,e)\ee
\end{thm}

\subsection{Norms, Transference, and Pseudo-random Measures}

First we introduce the $d$-dimensional box norms. We actually
introduce one norm for each linearly independent set of vectors \newline $\{e_1,...,e_d\}\subs\Z_N^d$.

For a function $f:\Z_N^{d}\rightarrow \mathbb{C}$
this norm with respect to a basis $e$ is given by\[
||f||_{\Box(e)^d}^{2^d} =
\mathbb{E}(\prod_{\omega\in\{0,1\}^d}f(x+\omega t
e):x\in\Z_N^d,\ t\in\Z_N^d)\] with the notation
$\omega t e=\omega_1t_1e_1+...+\omega_dt_de_d.$

That this is actually norm is not immediate, but for the standard basis it can be shown by repeated applications of the Cauchy-Schwarz inequality, similarly as for the Gowers norms (see for example  \cite{G1}). For a different basis, note that
we have $||f||_{\Box(e)^d}=||f\circ T||_{\Box^d}$ for an appropriate linear transformation $T$, where $||f||_{\Box^d}$ is the norm with
respect to the standard basis. The same way one shows \cite{G1} that the analogue of the so-called Gowers-Cauchy-Schwarz inequality holds

\begin{prop} ( $\Box^d(e)$-Cauchy-Schwarz inequality)

Given $2^d$ functions, indexed by elements of $\{0,1\}^d$, we have
\[
\langle
f_\omega:\omega\in\{0,1\}^d\rangle=\mathbb{E}(\prod_{\omega\in
\{0,1\}^d}f_{\omega}(x+\omega t
e):x\in\Z_N^d,t\in\Z_N^d)\leq \prod_{\omega\in
\{0,1\}^d}||f_\omega||_{\Box(e)^d}\]
\end{prop}

Gowers presents an alternative approach to the Green-Tao
Transference Theorem from a more functional analytic
point of view, making use of the Hahn-Banach Theorem. The specific
version he provides will be presented below after we
recall some definitions. First we note that $||\cdot||^*$ is
the defined to be the dual norm of $||\cdot||$.

\begin{defin}\label{QAP}
Let $||\cdot||$ be a norm on $\mathcal{H}=L^2(\mathbb{Z}_n)$ such
that $||f||_{L^{\infty}}\leq||f||^*$, and let
$X\subseteq\mathcal{H}$ be bounded. Then $||\cdot||$ is a
\textit{quasi algebra predual (QAP) norm} with respect to $X$ if
there exists an operator
$\mathcal{D}:\mathcal{H}\rightarrow\mathcal{H}$, a positive function $c$ on $\mathbb{R}$ and an increasing positive function $C$ on $\mathbb{R}$ satisfying:

\textit{(i)} $\langle f,\mathcal{D}f\rangle\leq1$ \textit{for all}
$f\in X$,

\textit{(ii)} $\langle f,\mathcal{D}f\rangle\geq c(\epsilon)$
\textit{for every} $f\in X$ \textit{with} $||f||\geq \epsilon$, and

\textit{(iii)} $||\mathcal{D}f_1...\mathcal{D}f_K||^*\leq C(K)$
\textit{for any} $f_1,...,f_K\in X$.
\end{defin}

This definition in enough to state the transference principle.

\begin{thmB} \label{transference} (Gowers \cite{G2})
Let $\mu$ and $\omega$ be non-negative functions on $Y,$ $Y$ finite,
with $||\mu||_{L^1},||\omega||_{L^1}\leq1$, and $\eta,\delta>0$ be
given parameters. Also let $||\cdot||$ be a QAP norm with respect to
$X$, the set of all functions bounded above by $\max\{\mu,\omega\}$ in
absolute value. There exists an $\epsilon$ such that the following
holds: If we have that $||\mu-\omega||<\epsilon$, then for every
function with $0\leq f\leq\mu$ there exists a function $g$ with
$0\leq g\leq \omega/(1-\delta)$ and $||f-g||\leq \eta.$
\end{thmB}

\noindent\textbf{Remarks:} By a simple re-scaling of the norms the constants 1 in Definition \ref{QAP} and Theorem B can be replaced by any other fixed constants. The actual form given by Gowers is more
explicit, in fact giving a specific choice of $\epsilon$. However,
for our purposes, we only need such an $\epsilon$ that is independent
of the size of $Y$. Also, for our purpose one may choose $\omega\equiv 1$ and $\de=1/2$.

%%%%%%%%%%%%%%%%%%%%%%%%%%%%%%%%%%%%%%%%%%

The definition of a pseudo-random measure in this paper will
be slightly stronger than that of Green and Tao, adapted to the higher dimensional settings. Let us begin with the one dimensional case. Following \cite{GT}, we define a \textit{measure} to be a
function $\nu:\Z_N\rightarrow \mathbb{R}$ to be a
non-negative function such that \[
\mathbb{E}(\nu(x):x\in\Z_N)=1+o(1).\] where the $o(1)$ notation means a quantity which tends to 0 as $N\to\infty$. A measure will be deemed
pseudo-random if it satisfies two properties at a specific level.
The first of these is known as the linear forms condition, as we will use only forms with integer coefficients we need a slightly simplified version.

\begin{defin} \label{forms} (Green-Tao \cite{GT})
Let $\nu$ be a measure, and $m_0,t_0\in\mathbb{N}$ be small
parameters. Then $\nu$ satisfies the $(m_0,t_0)$-\textit{linear
forms condition} if the following holds. For $m\leq m_0$ and $t\leq
t_0$ arbitrary, suppose that $\{L_{i,j}\}_{1\leq i\leq m, 1\leq
j\leq t}$ are integers, and that $b_i$ are arbitrary elements
of $\Z_N$. Given $m$ linear forms
$\phi_i:\Z_N^t\rightarrow \Z_N$ with \[
\phi_i(x)=\sum_{j=1}^t L_{i,j}x_j + b_i,\]
$x=(x_1,...,x_t)$ and $b=(b_1,...,b_t)$, if we have that each
$\psi_i$ is nonzero and that they are pairwise linearly independent,
then \eq\mathbb{E}\,\left(\prod_{i=1}^m \nu(\phi_i(x)):x\in\Z_N^t\right)=1+o(1),\ee where
the $o(1)$ term is independent of the choice of the $b_i$'s.
\end{defin}

The next condition is referred to as the correlation condition.

\begin{defin}\label{correlation}
Let $\nu$ be a measure. Then $\nu$ satisfies the $(m_0,m_1)$
correlation condition if for every $1\leq m\leq m_0$ there exists a function $\tau=\tau_m:\Z_N\to\R_+$ such that for all $k\in\N$
\[ \mathbb{E}(\tau^k(x):\,x\in\Z_N)=O_{m,k}(1)\] and also
\eq
\mathbb{E\,}\left(\prod_{i=1}^{m_1}\prod_{j=1}^{m_0}
\nu(\phi_i(y)+h_{i,j}):y\in\Z_N^r\right)\leq
\prod_{i=1}^{m_0}\left(\sum_{1\leq j<j'\leq m_0}\tau(h_{i,j}-h_{i,j'})\right)\ee
where the functions $\phi_i:\Z_N^r\rightarrow\Z_N$ are pairwise independent linear forms.
\end{defin}

\noindent\textbf{Remarks:}

This is a stronger condition that what is used in \cite{GT}, in fact they used the special case when
$m_1=1$, and $\phi$ is the identity. We define below a  $d$-pseudo-random measure to be a measure
satisfying these conditions at specific levels.

\begin{defin}\label{d-random}
We call a measure $\nu$ a \textit{$d$-pseudo-random} if
 if $\nu$ satisfies the \newline $((d^2+2d)2^{d-1},\,2d^2+d)$-linear
forms condition and the $(d,2^d)$-correlation condition
\end{defin}

We will deal with $d$-fold tensor product of measures, $\mu=\otimes_{i=1}^d \nu$ and call them $d$-measures. We will call such a $d$-measure $\mu$ to be pseudo-random if the corresponding measure $\nu$ is $d$-pseudo-random. Finally, note that for a $d$-measure
\[
\mathbb{E}(\mu(x):x\in\Z_N^d)=\prod_{i=1}^d\mathbb{E}
(\nu(x_i):x_i\in\Z_N)=1+o(1).\]

%%%%%%%%%%%%%%%%%%%%%%%%%%%%%%%%%%%%%%%%%%%%%%%%%%%%%%%%%%%

\subsection{Outline of the Paper.}

In Sections 2-3 we prove two key propositions, the so-called \emph{generalized von Neumann inequality} and the \emph{dual function estimate}. The first roughly says that the number of constellations defined by a set $e$ is controlled by the appropriate box norm. The second is the essential step in showing that the box norms are $QAP$ norms.

In Section 4, we prove our main results assuming that the measure exhibited in \cite{GT} is also $d$-pseudo-random in the sense defined above. First we show Theorem \ref{finite}, which follows then easily from the Transference Principle, that is from Theorem B. Next, we prove Theorem \ref{prime} by a standard argument passing from $\Z_N$ to $\Z$.

Finally, in an Appendix, we prove $d$-pseudo-randomness of the measure $\nu$ used by Green and Tao, slightly modifying their arguments of Sec.10 in \cite{GT} based on earlier work of Goldston and Yildrim \cite{GY1} \cite{GY2}.

\section{The Generalized von Neumann inequality.}

Let $e=\{e_1,\ldots,e_d\}\subs\Z_N^d$ be a base of $\Z_N^d$ which is also in general position, which in this settings means that $|\pi_i (e\cup\{0\})|=d+1$ for each $i$ where $\pi_i:\Z_N^d\to\Z_N$ is the orthogonal projection to the $i$-th coordinate axis.

\begin{prop}\label{neumann} (Generalized von Neumann Inequality)

Let $w=otimes^d\nu$ be a pseudo-random $d$-measure.
Given a function $0\leq f\leq w$, we have that
\eq\Lambda f:=
\mathbb{E}\,(f(x)f(x+te_1)...f(x+te_d):\,x\in\Z_N^d,t\in\Z_N) = O(||f||_{\Box(e')^{d}})\ee where $e'=\{e_d,e_d-e_1,...,e_d-e_{d-1}\}$.
\end{prop}

\begin{proof}
We shall apply the Cauchy-Schwartz inequality several times. Begin
by writing \[ \Lambda f\equiv\Lambda
=\mathbb{E}(f(x)\prod_{i=1}^df(x+t_1e_i):x\in\Z_N^{d},t_1\in\Z_N).\]
Push through the summation on $t_1$ and split the $f$ to write this
as\[
\mathbb{E}(\sqrt{f(x)}\,\mathbb{E}(\sqrt{f(x)}\prod_{i=1}^df(x+t_1e_i):t_1\in\Z_N):x\in\Z_N^{d}).\]
Applying Cauchy-Schwartz to get\[
\Lambda^2\leq\mathbb{E}(w(x)\prod_{i=1}^df(x+t_1e_i)\prod_{j=1}^df(x+t_1e_j+t_2e_j):t_1,t_2\in\Z_N,x\in\Z_N^{d}),\]
where we have made the substitution $t_2\mapsto t_1+t_2$ for the new
variable. Note that there should be a $\EE(w(x))=1+o(1)$ multiplier, following
from the fact that $f\leq w$ and from the linear forms condition, but for convenience we suppress it and will continue to do
so (this is a big O result, so this is not of any consequence). We
make one further substitution, $x\mapsto x-t_1e_1$, yielding\[
\Lambda^2\leq\mathbb{E}(w(x-t_1e_1)\prod_{i=2}^d
\prod_{\omega\in\{0,1\}}f(x+t_1e_i^{(1)}+\omega t^{(1)}e_i)
\prod_{\omega '\in\{0,1\}}f(x+\omega
't^{(1)}e_1):t_1,t_2\in\Z_N,x\in\Z_N^{d}),\] where
we have introduced the notations $e_i^{(j)}=e_i-e_j$, and
$t^{(i)}=\{t_{1+j}\}_{j=1}^i$. Note that the final product of this expression is independent of $t_1$.

We now repeat this procedure exactly, pushing through the $t_1$ sum
and splitting the terms independent of $t_1$, followed by a change
of variables. After $l$ applications of Cauchy-Schwarz inequality, we claim to have
\[ \Lambda^{2^l}\leq\mathbb{E}(W_l(x,t_1,...,t_{l+1})
\prod_{i=l+1}^d
\prod_{\omega\in\{0,1\}^l}f(x+t_1e_i^{(l)}+\omega t^{(l)}e_{i;l}))\]
\eq \times\prod_{\omega '\in\{0,1\}^{l}}f(x+\omega
't^{(l)}e_{l;l-1}):t_1,...,t_{l+1}\in\Z_N,x
\in\Z_N^{d}),\ee
\noindent for an appropriate weight function $W_l$ which is a product of
$w$'s, evaluated on linear forms which are pairwise linearly independent.

The notations introduced here are
$e_{i;l}=\{e_i,e_i^{(1)},...,e_i^{(l-1)}\}$ (note that $l>1$), and
$\omega t^{(l)}e_{i;l}=\omega_1t_2e_i+
\omega_2t_3e_i^{(1)}+...+\omega_lt_{l+1}e_i^{(l-1)}$.

To check this
form, using induction, apply the Cauchy-Schwarz inequality one more time with the new variable $t_1+t_{l+2}$ to
get
\[
\Lambda^{2^{l+1}}\leq\mathbb{E}(W_l(x,t_1,...,t_{l+1})W_l(x,t_1+t_{l+2},...,t_{l+1})\]
\[\times\prod_{i=l+1}^d
\prod_{\omega\in\{0,1\}^l}f(x+t_1e_i^{(l)}+\omega
t^{(l)}e_{i;l})f(x+t_1e_i^{(l)}+t_{l+2}e_i^{(l)}+\omega
t^{(l)}e_{i;l})
\]\[\times\prod_{\omega '\in\{0,1\}^{l}}w(x+\omega
't^{(l)}e_{l;l-1}):t_1,...,t_{l+2}\in\Z_N,x\in\Z_N^{d}).\]
Write
\eq
W_{l+1}'(x,t_1,...,t_{l+2})=W_l(x,t_1,...,t_{l+1})W_l(x,t_1+t_{l+2},...,t_{l+1})
\prod_{\omega
'\in\{0,1\}^{l}}w(x+\omega 't^{(l)}e_{l;l-1}).\ee
 We now apply the
substitution $x\mapsto x-t_1e_{l+1}^{(l)}$, note that
$e_i^{(l)}-e_{l+1}^{(l)}=e_i^{(l+1)}$, and set
\eq W_{l+1}(x,t_1,...,t_{l+2})=W_{l+1}'(x-t_1e_{l+1}^{(l)},t_1,...,t_{l+2}),\ee
This gives
\[
\Lambda^{2^{l+1}}\leq\mathbb{E}(W_{l+1}(x,t_1,...,t_{l+2})\times\prod_{i=l+2}^d
\prod_{\omega\in\{0,1\}^{l+1}}f(x+t_1e_i^{(l+1)}+\omega
t^{(l+1)}e_{i;l+1})
\]\[\times\prod_{\omega '\in\{0,1\}^{l+1}} f(x+\omega
't^{(l+1)}e_{l+1;l}):t_1,...,t_{l+2}\in\Z_N,x\in\Z_N^{d}).\]
and this is the form we wanted to obtain.

After $d-1$ iterations, one arrives at the form
\[
\Lambda^{2^{d-1}}\leq\mathbb{E}(W_{d-1}(x,t_1,...,t_d)\prod_{\omega
'\in\{0,1\}^{d}} f(x+\omega
't^{(d-1)}e_{d;d-1}):t_1,,...,t_d\in\Z_N,x\in\Z_N^{d}).\]
This may be written as
\[\Lambda^{2^{d-1}}\leq\mathbb{E}(\prod_{\omega
'\in\{0,1\}^{d}} f(x+\omega
't^{(d-1)}e_{d;d-1}):t_2,...,t_d\in\Z_N,x\in\Z_N^{d})+E,\]
where \[ E=\mathbb{E}((W_{d-1}(x,t_1,...,t_d)-1)\prod_{\omega
'\in\{0,1\}^{d}} f(x+\omega
't^{(d-1)}e_{d;d-1}):t_1,...,t_d\in\Z_N,x\in\Z_N^{d}).\]
To see that the main term is in fact an appropriate box norm, notice
that\[ e_{d;d-1}=\{e_d,e_d-e_1,...,e_d-e_{d-1}\}\] is also in
general position.

To deal with the error term $E$, we apply the Cauchy-Schwarz inequality one more time to get
\[E\leq\mathbb{E}((W(x,t_2,...,t_d)-1)^2\prod_{\omega
'\in\{0,1\}^{d}} w(x+\omega
't^{(d)}e_{d;d-1}):t_2,,...,t_{d+1}\in\Z_N,x\in\Z_N^{d}),\]
where we have set
\[
W(x,t_2,...,t_d)=\mathbb{E}(W_{d-1}(x,t_1,t_2,...,t_d):t_1\in\Z_N)\]
and again used the fact that $f\leq w$. Now to show that $E=o(1)$,
it is enough to show that the linear forms defining $W$ are
pairwise independent, after of course expanding $(W-1)^2$ and
applying the linear forms condition. By following the construction
of $W$, this amounts to showing that at each step $W_l$ satisfies
pairwise independence, which itself reduces to showing that the
coefficient of $x$ is 1 in each form and each form has a nonzero
coefficient in $t_1$ (in each coordinate).

To be more precise, the case $l=1$ is immediate. Assuming
this is so for $l$ fixed, then
\[W_{l+1}'(x,t_1,...,t_{l+2})=W_l(x,t_1,...,t_{l+1})W_l(x,t_1+t_{l+2},...,t_{l+1})\prod_{\omega
'\in\{0,1\}^{l}}w(x+\omega 't^{(l)}e_{l;l-1}).\] certainly satisfies
this, as the the forms in $W_l(x,t_1,...,t_{l+1})$ and
$W_l(x,t_1+t_{l+2},...,t_{l+1})$ are pairwise independent because
the $t_1$ coefficient is non-zero, and $\prod_{\omega
'\in\{0,1\}^{l}}w(x+\omega 't^{(l)}e_{l;l-1})$ is independent of
$t_1$. The statement about the coefficient of $x$ is obvious. Also,
it not hard to see that the vector multiple of $t_1$ is either
$e_{l+1}$ or $e_{l+1}^{(i)}$ (for forms appearing after $i$
applications of Cauchy-Schwarz). Thus the statement is true for $l+1$.

The fact that $E=o(1)$ then follows directly from the $(d(d+2)2^{d-1},d(2d+1))$ linear forms condition.
\end{proof}

\section{The dual function estimate.}

As before we assume that a basis $e=\{e_1,...,e_d\}\subs \Z_N^d$ is given which is in general position.
We will use the notation $\omega y e= \omega_1 y_1 e_1+...+\omega_d y_d e_d$, for $\omega\in\{0,1\}^d$ and  $y\in\Z_N^d$. First we define the dual of a function $f:\Z_N^d\to\R$ with respect to the $\|\ \|_{\Box (e)^d}$ norm.

\begin{defin}. Let $f:\Z_N^d\to\R$ be a given function and let $e=\{e_1,...,e_d\}\subs \Z_N^d$ be a basis of $\Z_N^d$. The dual of the the function $f$ is the function
\eq \mathcal{D}f (x) = \EE\, (\prod_{\omega\in\{0,1\}^d,
\, \omega\neq0} f(x+\omega t e):t\in\Z_N^d)\ee
\end{defin}

\begin{prop}\label{dual}
With $X$ and $\mathcal{D}$ as above, and $e$ in general position, we
have\[||\mathcal{D}f_1...\mathcal{D}f_K||_{\Box(e)^d}^*\leq C(K)\]
for any $f_1,...,f_K\in X$.
\end{prop}

\begin{proof}
We must show that \[ \langle
f,\mathcal{D}f_1...\mathcal{D}f_K\rangle\leq C_K ||f||_{\Box(e)^d}\]
by the definition of the dual norm. By applying the definition of
$\mathcal{D}f$, the LHS gives
\[ \langle
f,\mathcal{D}f_1...\mathcal{D}f_K\rangle=\mathbb{E}(f(x)\prod_{i=1}^K\mathbb{E}(\prod_{\omega\in\{0,1\}^d,
\, \omega\neq0}f_i(x+\omega
t^ie):t^i\in\Z_N^d):x\in\Z_N^d).\] Expanding out the
products then gives the RHS as\[
\mathbb{E}(\mathbb{E}(f(x)\prod_{\omega\in\{0,1\}^d, \,
\omega\neq0}\prod_{i=1}^Kf_i(x+\omega t^ie+\omega
te):x,t\in\Z_N^d):T=(t^1,...,t^K)\in(\Z_N^d)^K)\]
after a substitution $t^i\mapsto t+t^i$ for each $i$ for some fixed
$t$, and adding a redundant summation in $t$. Now we call $
F_{(\omega,T)}(x)=\prod_{i=1}^Kf_i(x+\omega t^ie)$ for non-zero
$\omega$, and  $F_{(0^d,T)}(x)=f(x)$. The last expression then
becomes\[ \mathbb{E}(\langle
F_{(\omega,T)}:\omega\in\{0,1\}^d\rangle:T\in\Z_N^d).\] By
applying the $\Box(e)$-Cauchy-Schwarz inequality, we have arrived at
\[ ||\mathcal{D}f_1...\mathcal{D}f_K||_{\Box(e)^d}^*\leq \mathbb{E}(
\prod_{\omega\in\{0,1\}^d,\, \omega\neq
0^d}||F_{(\omega,T)}||_{\Box(e)^d}:T\in\Z_N^d).\] An
application of the Holder inequality gives that the RHS is bounded
above by\[\prod_{\omega\in\{0,1\}^d,\, \omega\neq 0^d}\mathbb{E}(
||F_{(\omega,T)}||_{\Box(e)^d}^{2^d}:T\in(\Z_N^d)^K),\]
where we added one factor of the constant 1 function, which has
$L^q$-norm one for each $q$. Thus, we now just need to show that for
a fixed $\omega\neq0^d$ we have\[ \mathbb{E}(
||F_{(\omega,T)}||_{\Box(e)^d}^{2^d}:T\in(\Z_N^d)^K)=O(K)\]
for $T=(t^1,...,t^K)$.

We continue by expanding the last expression for a fixed
$\omega\neq 0^d$,
\[
||F_{(\omega,T)}||_{\Box(e)^d}^{2^d}:T\in(\Z_N^d)^K)=O(K)=\mathbb{E}(\prod_{\omega'\in\{0,1\}^d}\prod_{i=1}^Kf_i(x+\omega
t^ie+\omega'te):x,t,t^1,...,t^K\in\Z_N^d).\] The RHS
factorizes as\[
\mathbb{E}(\prod_{i=1}^K\mathbb{E}(\prod_{\omega'\in\{0,1\}^d}f_i(x+\omega
ye+\omega'te):y\in\Z_N^d):x,t\in\Z_N^d).\] Applying
the bound $f\leq\nu$ gives
\[\mathbb{E}(\mathbb{E}^K(\prod_{\omega'\in\{0,1\}^d}\nu(x+\omega
ye+\omega'te):y\in\Z_N^d):x,t\in\Z_N^d).\]

The inner sum is now split component wise \[
\mathbb{E}(\prod_{j=1}^d\prod_{\omega'\in\{0,1\}^d}\mu((\omega
ye)_j+(\omega'te+x)_j):y\in\Z_N^d),\] where the notation
$(x)_j$ denotes the $j^{th}$ coordinate. The terms $(\omega ye)_j$
represent the linear forms $\sum_{s=1}^d\omega_sy_s(e_s)_j$, which
satisfy the hypothesis in the $(d,2^d)$ correlation condition by the
assumptions on $e$. Hence we have
\[\mathbb{E}(\prod_{j=1}^d\prod_{\omega'\in\{0,1\}^d}\mu((\omega
ye)_j+(\omega'te+x)_j):y\in\Z_N^d)\leq
\prod_{j=1}^d\sum_{\omega'\neq\omega''}\tau(((\omega'-\omega'')te)_j),\]
as the $(x)_j$ terms drop out in the subtraction.

Plugging this bound back in gives\[
\mathbb{E}((\prod_{j=1}^d\sum_{\omega'\neq\omega''}\tau(((\omega'-\omega'')te)_j))^K:t\in\Z_N^d).\]
Making use of the triangle inequality in $\mathcal{L}^{dK}$, after
another application of Holder, reduces our task to bounding
\[
\prod_{j=1}^d\sum_{\omega'\neq\omega''}\mathbb{E}(\tau^{dK}(((\omega'-\omega'')te)_j):t\in\Z_N^d).\]
By the assumptions on $e$ and the fact that
$\omega'-\omega''\neq0^d$, $((\omega'-\omega'')te)_j$ provides a
uniform cover of $\Z_N$, and we may reduce this to \[
\mathbb{E}(\tau^{dK}(t):t\in\Z_N).\] This expression is
$O_K(1)$.
\end{proof}

\section{Proof of the main results.}

In this section we prove our main results under the assumption that the measure exhibited in \cite{GT} is $d$-pseudo-random, i.e. it satisfies Definition \ref{d-random}.

\subsection{Proof of Theorem \ref{finite}.}
Let $e=\{e_1,\ldots,e_d\}\subs\Z_N^d$ be a basis which is in general position. For a function
$f:\Z_N^d\to\R$ we define its dual by
\eq\mathcal{D}f(x)=\mathbb{E}(\prod_{\omega\in\{0,1\}^d,\omega\neq
0}f(x+\omega t e):t\in\Z_N^d).\ee
Then clearly
\eq \langle f,Df\rangle = \|f\|_{\Box (e)^d}^{2^d}\ee

Let $\mu=\otimes^d \nu$ be a pseudo-random $d$-measure, and let $X$ be the set of functions $f$ on $\Z_N^d$ such that $|f|\leq \mu$ pointwise.

\begin{lemma} The norm $\|\ \|_{\Box (e)^d}$ is a quasi algebra predual (QAP) norm, with respect to the set $X$ and the operator $D$.
\end{lemma}

\begin{proof} We have already shown part (iii) of Definition \ref{QAP}, which was the content of Proposition \ref{dual}. If $\|f\|_{\Box (e)^d}^d\leq\eps$ then
\[\langle f,Df\rangle = \|f\|_{\Box (e)^d}^{2^d}\leq\eps^{2^d}\]
and part (ii) follows. Finally, since $|f|\leq \mu$ it follows
\[\langle f,Df\rangle \leq \|\mu\|_{\Box (e)^d}^{2^d} = 1+o(1)\]
as the linear forms $(x+\omega t e)_j$ are pairwise linearly independent (for each $j$) and $\nu$ satisfies the linear forms condition.
\end{proof}

We are in the position to apply the transference principle to decompose a function $0\leq f\leq \mu$ into the sum of a bounded function $g$ and a function $h$ which has small contribution to the expression in (1.2).

\begin{proof}[Proof of Theorem \ref{finite}] Let $\al>0$ and let $0\leq f\leq \mu$ be function such that $\EE f\geq\al$, where $\mu$ is a pseudo-random $d$-measure on $\Z_N^d$. We apply Theorem B, with $Y=\Z_N^d$, $\de=1/2$ and $\eta>0$. Note that since $\mu$ is a measure one has that $\|\mu\|_{L^1}=\EE\mu=1+o(1)$. Since $\|\ \|_{\Box (e)^d}$ is a \emph{QAP} norm with respect to the set $X=\{f:Y\to\R,\ |f|\leq\mu\}$, it follows that there is an $\eps>0$ such that if \eq\|\mu-1\|_{\Box (e)^d}<\eps\ee then there is a decomposition $f=g+h$ such that
\eq 0\leq g\leq 2\ \ \ \ \ \ \text{and}\ \ \ \ \ \ \|h\|_{\Box (e)^d}<\eta.\ee
Since $\mu$ is pseudo-random $\|\mu-1\|_{\Box (e)^d}=o(1)$  thus (4.3) holds for large enough $N$. Using this decomposition together with Theorem A and Proposition \ref{neumann} one may write
\[\mathbb{E}(f(x)f(x+te_1)...f(x+te_d):x\in\mathbb{Z}_N^d,t\in\mathbb{Z}_N) =\]
\[=\ \mathbb{E}(g(x)g(x+te_1)...g(x+te_d):x\in\mathbb{Z}_N^d,t\in\mathbb{Z}_N)\ +\  O(\|h\|_{\Box (e)^d}) \geq c'(\al,e)-C_d \eta \geq c'(\al,e)/2\]
by choosing $\eta$ sufficiently small with respect to $\al$ and $e$. This proves Theorem \ref{finite}.
\end{proof}

\subsection{Proof of Theorem \ref{prime}.} Let us identify $[1,N]^d$ with $\Z_N^d$. First we show that constellations in $\Z_N^d$ defined by $e$ which are contained in a box $B\subs [1,N]^d$ of size $\eps N$, are in fact genuine constellations contained in $B$. We say that $e=\{e_1,\ldots,e_d\}\in\Z^{d^2}$ is \emph{primitive} if the segment $[0,e]$ does not contain any other lattice points other than its endpoints in $\Z^{d^2}$ considered as a lattice point in $\Z^{d^2}$.
Let us also define the positive quantity $\tau(e)$ by
\[\tau (e)=\inf_{m\notin\{0,e\},\,x\in [0,e]} |m-x|_\infty\ \ \ \ \text{where}\ \ \ \ |x|_\infty=\max_{1\leq d^2} |x_j|\]
$m$ is running through the lattice points $\Z^{d^2}$ other than $0$ and $e$.

\begin{lemma} \label{modN} Let $0<\eps<\tau(e)$. Let $N$ be sufficiently large, and let $B=I^d$ be a box of size $\eps N$ contained in $[1,N]^d\simeq\Z_N^d$.
If there exist $x\in\Z_N^d$ and $t\in\Z_N\backslash\{0\}$ such that $x\in B$ and $x+te\subs B$ as a subset on $\Z_N^d$, then there exists a scalar $t'\neq 0$ such that $x+t'e\subs B$ also as a subset of $\Z^d$.
Moreover if $e$ is primitive (and $1\leq t<N$) then one may take $t'=t$ or $t'=t-N$.
\end{lemma}

\begin{proof} First, note that one can assume $e$ is primitive as $x+te=x+tse'$ for a fixed primitive $e'$ and $s\in\N$. By our assumption, there is an $x\in [1,N]^d$ and $t\in [1,N-1]$ such that $x\in B$ and $x+te_j\in B+(N\Z)^d$ for all $1\leq j\leq d$. Thus for each $j$ there exits $m_j\in\Z^d$ such that $|te_j-Nm_j|_\infty \leq\eps N$ and hence $|\la e-m|_\infty\leq\eps$, where $m=\{m_1,\ldots,m_d\}\in\Z^{d^2}$ and $\la=t/N$. Since $0<\la<1$ and $\eps<\tau(e)$ this implies that $m=0$ or $m=e$. If $m=0$ then $|te|_\infty \leq\eps N$ and since $x\in B$ it follows that $x+te\subs B\subs\Z^d$. If $m=e$ then $|(t-N)e_j|_\infty\leq\eps N$ thus $x+(t-N)e\subs B\subs\Z^d$, so $x+t'e\subs B$ as a subset of $\Z^d$.
This proves the lemma.
\end{proof}

Let us briefly recall the pseudo-random measure $\nu$ defined in Sec.9 \cite{GT}. Let $w=w(N)$ be a sufficiently slowly growing function (choosing $w(N)\ll \log\log\,N$ is sufficient as in \cite{GT}) and let $W=\prod_{p\leq w} p$ be the product of primes up to $w$. For given $b$ relative prime to $W$ define the modified von Mangoldt function $\bar{\La}_b$ by
\eq \bar{\La}_b (n)= \left\{ \begin{array}{ll}
         \frac{\phi(W)}{W}\,\log (Wn+b) & \mbox{if $Wn+b$ is a prime};\\
         0 & \mbox{otherwise}.\end{array} \right. \ee
where $\phi$ is the Euler function. Note that by Dirichlet's theorem on the distribution of primes in residue classes one has that $\sum_{n\leq N} \bar{\La}_b (n)=N (1+o(1))$. Also, if $A\subs\PP^d$ is of positive relative $\al$ and if $\bar{\La}_b^d:=\otimes^d \bar{\La}_b$ is the $d$-fold tensor product of $\bar{\La}_b$ the it is easy to see that there exists a $b$ such that
\eq \limsup_{N\to\infty} N^{-d} \sum_{x\in [1,N]^d} \1_A (x) \bar{\La}_b^d (x)\,>\,\al/2\ee
We will fix such $b$ and choose $N$ sufficiently large $N$ for which the expression in (4.6) is at least $\al/2$. Let $R=N^{d^{-1}2^{-d-5}}$ and recall the Goldston-Yildirim divisor sum \cite{GT}, \cite{GY1}
\[\La_R(n)=\sum_{d|n,d\leq R} \mu(d)\,\log(R/d)\] $\mu$ being the Mobius function. For given small parameters $0<\eps_1<\eps_2<1$ (whose exact values will be specified later) recall the Green-Tao measure
\eq \nu(n)= \left\{ \begin{array}{ll}
         \frac{\phi(W)}{W}\,\frac{\La_R(Wn+b)^2}{\log\,R} & \mbox{if $\eps_1N\leq n\leq\eps_2 N$};\\
         1 & \mbox{otherwise}.\end{array} \right. \ee
Note that $\nu(n)\geq 0$ for all $n$, and also it is easy to see that
for $N$ sufficiently large, one has that
\eq\nu(n)\geq d^{-1}2^{-d-6}\,\bar{\La}_b (n)\ee for all $\eps_1 N\leq n\leq \eps_2 N$.
Indeed, this is trivial unless $Wn+b$ is a prime. In that case, since $\eps_1 N>R$, $\La_R(Wn+b)=\log\,R=d^{-1}2^{-d-5}\log\,N$ and (4.8) follows.

\begin{proof} [Proof of Theorem \ref{prime}.] Set $\mu =\otimes^d \nu$, and let
\eq g(x):= c_d\,\bar{\La}_b^d (x)\,\1_A(x)\,\1_{[\eps_1 N,\eps_2 N]^d} (x)\ \ \ \ \ \ (c_d=d^{-d}2^{-d^2-6d})\ee
Then by (4.8) one has that $g(x)\leq\mu(x)$ for all $x\in\Z_+^d$.
By (4.6) one may choose a sufficiently large number $N'$ for which
\eq (N')^{-d}\sum_{x\in [1,N']^d} \1_A (x) \bar{\La}_b^d (x)\,>\,\al/2\ee
and a prime $N$ such that
\[ (1-\frac{\al}{100 d})N'\leq\eps_2 N \leq N'\]
If $\eps_1$ is such that $\eps_1/\eps_2\leq \al/100 d$, then by the Prime Number Theorem in arithmetic progressions
\eq (N')^{-d} \sum_{x\in [1,N']^d\backslash [\eps_1N,\eps_2N]^d} \bar{\La}_b^d (x)\,\leq\, \al/10\ee
It follows from (4.10) and (4.11)
\eq
N^{-d}\sum_{x\in [1,N']^d} g(x)\ \geq\ c_d\, N^{-d} \sum_{x\in [\eps_1 N,\eps_2 N]^d} \1_A(x) \bar{\La}_b^d (x)\ \geq\  c_d \eps_2^d \al/4
\ee

Using the identification $[1,N]^d\simeq \Z_N^d$, one has that $\EE(g(x):\,x\in\Z_N^d)\geq\al'$ (\ with $\al'=c_d^d\eps_2^d\al/4$), and $0\leq g(x)\leq \mu (x)$ for all $x$. Thus, save for proving that the measure $\nu$ is $d$-pseudo-random, Theorem \ref{finite} implies that
\[\EE(g(x)g(x+te_1)\ldots g(x+te_d):\ x\in\Z_N^d,t\in\Z_N)\,\geq\,c'(\al,e)>0.\]
Note that the contribution of trivial constellations, corresponding to $t=0$, is at most $O(N^{-1}\log^d N)$, as $|\bar{\La}_b^d|\leq \log^d N$ uniformly on $[1,N]^d$. Since the support of $g$ is contained in $A\cap [\eps_1 N,\eps_2 N]^d$, Lemma \ref{modN} implies that $A\cap [\eps_1 N,\eps_2 N]^d$ must contain genuine constellations of the form $\{x,x+te_1,\ldots,x+te_d\}$ as a subset of $\Z^d$. Choosing an infinite sequence of $N$'s it follows that $A$ contains infinitely many constellations defined by $e$.
\end{proof}

\section{Appendix: The correlation condition.}

To complete the proof of Theorem \ref{prime}, one needs to show that the measure $\nu$ defined in (4.7) satisfies both the linear forms conditions and the $(d,2^d)$ correlation conditions  given in (1.4). Since the measure $\nu$ is the same (apart from the slight change in the interval where $\nu\equiv 1$) is the one given in \cite{GT} (see Definition 9.3, there), the linear forms condition is already established in Prop. 9.8 in \cite{GT}. It turns out that the arguments given in \cite{GT} (see Prop. 9.6, Lemma 9.9 and Prop.9.10) generalize in a straightforward manner to obtain the more general $(m_0,m_1)$ correlation condition for any given specific values of $m_0$ and $m_1$.

\begin{prop}\label{correlation}
For a fixed $m_0,m_1$, there
exists a function $\tau$ such that \[ \mathbb{E}\tau^k=O_k(1)\] and
also \eq \mathbb{E}(\prod_{i=1}^{m_1}\prod_{j=1}^{m_0}
\nu(\phi_i(y)+h_{i,j}):y\in\Z_N^r)\leq
\prod_{i=1}^{m_0}(\sum_{1\leq j<j'\leq m_0}\tau(h_{i,j}-h_{i,j'}))\ee
where the $\phi_i:\Z_N^r\rightarrow\Z_N$
are pairwise linearly independent linear forms.
\end{prop}

Let us first note that the arguments of Lemma 9.9 and Prop. 9.10 of \cite{GT} applies to our case and it is enough to establish the following inequality (see Prop. 9.6 \cite{GT})
\[\EE\,(\prod_{i=1}^{m_1}\prod_{j=1}^{m_0} \La_R^2 (W(\phi_i(y)+h_{i,j})+b):\ y\in B)\]
\eq\ \ \ \ \ \  \leq\ C_M \left(\frac{W\log\,R}{\phi(W)}\right)^M\ \prod_{i=1}^{m_1} \prod_{p|\triangle_i} (1+O_M(p^{-1/2}))\ee
where $M=m_1 m_0$ and $B$ is a box of size at most $R^{10M}$. Moreover one can assume that $h_{i,j}\neq h_{i,j'}$ for all $i$, $j\neq j'$.

The next step is, following \cite{GT}, to write the the expression
\[
\mathbb{E}(\prod_{i=1}^{M} \Lambda_R^2(\theta_i(y)):y\in B),\] where
 $\theta_i=W(\phi_{\lfloor
i/m_1\rfloor}(y)+h_{\lfloor i/m_1\rfloor,\,(i \,(p)})+b$ ($\lfloor
x\rfloor$ is the floor function, $i\, (m_1)$ is $i$ modulo $m_1$),
to as a contour integral of the the following form plus a small error
\eq (2\pi
i)^{-M}\int_{\Gamma_1}...\int_{\Gamma_1}F(z,z')\prod_{j=1}^M \frac{R^{z_j+{z'}_j}}{z_j^2 {z'}_j^2} dz_j dz'_j,\ee
where $z=(z_1,...,z_M)$, $z'=(z'_1,...,z'_M)$, and function $F(z,z')$ is taking form of an Euler product
\[F(z,z')=\prod_p E_p(z,,z'),\] where \[ E_p(z,z')= \sum_{X,X'\subseteq
[M]}\frac{(-1)^{|X|+|X'|}\omega_{X\bigcup X'}(p)}{p^{\sum_{j\in
X}z_j+\sum_{j\in X'}z'_j}}.\] The function $\omega$ relates this
expression to the particular forms. Specifically\[
\omega_X(p)=\mathbb{E}(\prod_{i\in X }\mathbf{1}_{\theta_i\equiv 0
\,(p)}:x\in\mathbb{Z}_N^r).\]

\begin{lemma} (Local factor estimate).
Set the intervals $I_i=[(i-1) m_1+1,im_1]$ as a partition of $[M]$.
For $\alpha\in I_i$, the homogeneous part of $\theta_{\alpha}$ is
$W\phi_i$. Also, set $\Delta_i=\prod_{j<j;\,j,j'\in
I_i}|h_{i,j}-h_{i,j'}|$.The following estimates hold: $\omega_X(p)$:
\begin{enumerate}
\item If $p\leq w(N)$, then $\omega_X(p)=1$ if $|X|=0$, and is 0
otherwise.

\item If $p> w(N)$ and $|X|=0$, then $w_X(p)=1$.

\item If $p> w(N)$ and $X\subseteq I_i$ is nonempty, we
have $w_X(p)=p^{-1}$ when $|X|=1$, and $w_X(p)\leq p^{-1}$ when
$|X|>1$. In the latter case, if $p\nmid\Delta_\alpha$, we have that
$\omega_X(p)=0$.

\item If $p> w(N)$ and $X\cap I_i\neq\emptyset$ and $X\cap I_{i'}\neq\emptyset$ for some $i\neq i'$, we
have $\omega_X(p)\leq p^{-2}$ .
\end{enumerate}
\end{lemma}

\begin{proof}
When $p\leq w(N)$, then $W \phi_i+b\equiv b \, (p)$, giving the
first result. The second statement is trivial.

For the third statement, let us start with $X\subseteq I_i$ with
$|X|=1$. Then we have\[
\mathbb{E}(\mathbf{1}_{W(\phi_i(y)+h_{i,\,j})+b\equiv 0 \,
(p)}:y\in\mathbb{Z}_N^r)=p^{-1}\] for any fixed $j$, proving the
first part. The second part requires an estimate of \[
\mathbb{E}(\mathbf{1}_{W(\phi_i(y)+h_{i,\,j})+b\equiv 0 \,
(p)}\mathbf{1}_{W(\phi_i(y)+h_{i,\,j'})+b\equiv 0 \,
(p)}:y\in\mathbb{Z}_N^r),\] with $j\neq j'$. If
$p|\,|h_{\alpha,\,j}-h_{\alpha,\,j'}|$, then the we are left with
simply a single equation ($p\nmid W$), and may refer to the first
part. When $p\nmid\Delta_\alpha$, $\omega_X(p)=0$ as $h_{i,\,j}$ is
not congruent to  $h_{i,\,j',}$ modulo $p$.

For the last statement, we have the upper bound \[
\mathbb{E}(\mathbf{1}_{W(\phi_i(y)+h_{i,\,j})+b\equiv 0 \,
(p)}\mathbf{1}_{W(\phi_i'(y)+h_{i',\,j'})+b\equiv 0 \,
(p)}:y\in\mathbb{Z}_N^r)\] for some $i\neq i'$ and $j,j'$. The forms
$\phi_i$ and $\phi_{i'}$ are linearly independent modulo $p$ (see
the proof of Lemma 10.1 in \cite{GT}), hence we have the
intersection of two distinct linear algebraic sets, which has size
at most $p^{r-2}$.
\end{proof}

The terms $E_p$ in the Euler product can be separated as \[
E_p(z,z')=1-\mathbf{1}_{p>w(N)}\sum_{j=1}^M(p^{-1-z_j}+p^{-1-z'_j}-p^{-1-z_j-z'_j})\]\[
+\sum_{i=1}^{m_1}\mathbf{1}_{p>w(N);\,
p|\Delta_i}\lambda^{(i)}_p(z,z')+\sum_{X\bigcup X'\nsubseteq
I_\alpha,\, \alpha\in[m_1];\, |X\bigcup
X'|>1}\frac{O_M(p^{-2})}{p^{\sum_Xz_j+\sum_{X'}z'_j}},\]

where
\[\lambda^{(i)}_p(z,z')=\sum_{X\bigcup X'\subset I_i;\, |X\bigcup
X'|>1}\frac{O_M(p^{-1})}{p^{\sum_Xz_j+\sum_{X'}z'_j}}.\]

We define the terms
\[E_p^{(0)}=1+\sum_{i=1}^{m_1}\mathbf{1}_{p>w(N);\,
p|\Delta_i}\lambda^{(i)}_p(z,z'),\] and factorize
$E_p=E_p^{(0)}E_p^{(1)}E_p^{(2)}E_p^{(3)}$ as follows:\[
E_p^{(1)}=\frac{E_p}{E_p^{(0)}\prod_{j=1}^M(1-\mathbf{1}_{p>w(N)}p^{-1-z_j})(1-\mathbf{1}_{p>w(N)}p^{-1-z'_j})(1-\mathbf{1}_{p>w(N)}p^{-1-z_j-z'_j})^{-1}}\]
\[
E_p^{(2)}=\prod_{j=1}^M(1-\mathbf{1}_{p\leq
w(N)}p^{-1-z_j})^{-1}(1-\mathbf{1}_{p\leq
w(N)}p^{-1-z'_j})^{-1}(1-\mathbf{1}_{p\leq w(N)}p^{-1-z_j-z'_j})\]
\[E_p^{(3)}=\prod_{j=1}^M(1-p^{-1-z_j})(1-p^{-1-z'_j})(1-p^{-1-z_j-z'_j})^{-1},\]
and set $G_i=\prod_p E_p^{(i)}$, noting that \[ G_3=
\prod_{j=1}^M\frac{\zeta(1+z_j+z'_j)}{\zeta(1+z_j)\zeta(1+z'_j)}.\]

The the following is the analogue of lemma 10.6 in \cite{GT}. To state it, Let us recall the domain $\mathcal{D}_\sigma^M$ to be the
set\[ \{ z_j,z'_j: \Re z_j,\Re z'_j\in(-\sigma,100)\, ,\, 1\leq
j\leq M\} .\] We also have the norms on for $f$ analytic on
$\mathcal{D}_\sigma^M$, denoted
$||f||_{\mathcal{C}^k(\mathcal{D}_\sigma^M)}$, given by\[
||f||_{\mathcal{C}^k(\mathcal{D}_\sigma^M)}=\sup||(\frac{\partial}{\partial
z_1})^{\alpha_1}... (\frac{\partial}{\partial
z_M})^{\alpha_1}(\frac{\partial}{\partial z'_1})^{\alpha_1}
...(\frac{\partial}{\partial
z'_M})^{\alpha_1}f||_{\mathcal{L}^{\infty}(\mathcal{D}_\sigma^M)},\]
where the supremum is taken over all
$\alpha_1,...,\alpha_M,\alpha'_1,...,\alpha'_M$ whose sum is at most
$k$.

\begin{lemma}
Let $0<\sigma=1/(6M)$. Then the Euler products $G_i$ are absolutely
convergent for $i=0,1,2$ in the domain $\mathcal{D}_\sigma^M$, and
hence represent analytic functions on this domain. We also have the
estimates\[
||G_0||_{\mathcal{C}^r(\mathcal{D}_\sigma^M)}=O_M(\log(R)/\log\log(R))^r\prod_{p|\prod_{i=1}^{m_1}\Delta_i}(1+O_M(p^{2M\sigma-1}))\]
\[||G_0||_{\mathcal{C}^M(\mathcal{D}_{1/6M}^M)}\leq
\exp(O_M(\log^{1/3}(R)))\]\[
||G_1||_{\mathcal{C}^M(\mathcal{D}_{1/6M}^M)}\leq O_M(1)\]\[
||G_2||_{\mathcal{C}^M(\mathcal{D}_{1/6M}^M)}\leq O_{M,w(N)}(1)\]\[
G_0(0,0)=\prod_{i=1}^{m_1}\prod_{p|\Delta_i}(1+O_M(p^{-1/2}))\]\[
G_1(0,0)=1+o_M(1)\]\[ G_2(0,0)=(W/\phi(W))^M,\] where the first
bound is for all $0\leq r\leq M$.

\end{lemma}

\begin{proof}
The estimates proceed exatctly as in Lemma 10.3 and Lemma 10.6 in \cite{GT} with
$\Delta=\prod_{i=1}^{m_1}\Delta_i$, barring the statement about
$G_0(0,0)$. To see this, we have \[ G_0(0,0)=
\prod_{p|\Delta}E_p^{(0)}=\prod_{p|\Delta}(1+\sum_{i=1}^{m_1}
\lambda_p^{(i)}(0,0))\leq \prod_{i=1}^{m_1}\prod_{p|\Delta_i}(1+
|\lambda_p^{(i)}(0,0)|)\]
 and we crudely have
$|\lambda_p^{(i)}(0,0)|=1+O_M(p^{-1/2})$.
\end{proof}

The expression in (5.3) takes the form

\[ (2\pi
i)^{-M}\int_{\Gamma_1}...\int_{\Gamma_1}G(z,z')\prod_{j=1}^M
\frac{\zeta(1+z_j+z'_j)R^{z_j+z_j'}}{\zeta(1+z_j)\zeta(1+z'_j)z_j^2z_j^{'2}}
dz_jdz'_j\] with $G=G_0 G_1 G_2$.  To estimate it let us recall the following general result on contour integration from \cite{GT}, see Lemma 10.4 there.

\begin{lemma}\label{int}(Goldston-Yildirim \cite{GT}\cite{GY2})
Let $R$ be a positive number. If $G(z,z')$ is analytic in the $2M$
variables on $\mathcal{D}_\sigma^M$ for some $\sigma>0$, and we have
the estimate \[
||G||_{\mathcal{C}^k(\mathcal{D}_\sigma^M)}=\exp(O_{M,\sigma}(\log^{1/3}(R))),\]
then \[(2\pi
i)^{-M}\int_{\Gamma_1}...\int_{\Gamma_1}G(z,z')\prod_{j=1}^M
\frac{\zeta(1+z_j+z'_j)R^{z_j+z_j'}}{\zeta(1+z_j)\zeta(1+z'_j)z_j^2z_j^{'2}}dz_jdz'_j\]\[
=G(0,...,0)\log^M(R)+\sum_{j=1}^MO_{M,\sigma}(||G||_{\mathcal{C}^j(\mathcal{D}_\sigma^M)})\log^{M-j}(R)+O_{M,\sigma}(\exp(-\delta
\sqrt{\log(R)}))\] for some $\delta>0$.
\end{lemma}

Estimate (5.2) follows easily applying Lemma 5 (with $\sigma =1/6M$) to $G=G_0G_1G_2$ using Lemma 4, which in turn implies
Proposition~\ref{correlation}, where the function $\tau$ is defined precisely as in \cite{GT}. This finishes the proof of Theorem \ref{prime}.

%%%%%%%%%%%%%%%%%%%%%%%%%%%%%%%%%%%%%%%%%%%%%%%%%%%%%%%%%%%%%%%%%%%

\end{document}